\documentclass{amsart}
\usepackage{amsfonts,amsmath, comment, amssymb, rotating, color,  graphicx}
\newtheorem{theorem}{Theorem}
\theoremstyle{plain}

\newtheorem{corollary}{Corollary}
\newtheorem{definition}{Definition}

\newtheorem{lemma}{Lemma}

\newtheorem{remark}{Remark}
\numberwithin{equation}{section}

\newcommand{\I}{{\mathcal I}}

\begin{document}
\title[Graphs in Projective Planes]{Graphs Embedded into Finite Projective Planes}
\author{Keith Mellinger}
\address{Department of Mathematics\\
University of Mary Washington\\
Fredericksburg, VA  22401.}
\email{kmelling@umw.edu}
\author{Ryan Vaughn}
\address{Department of Mathematics\\
University of Mary Washington\\
Fredericksburg, VA  22401.}
\email{vaughn.ryan@gmail.com}
\author{Oscar Vega}
\address{Department of Mathematics\\
California State University, Fresno \\ Fresno, CA 93740.}
\email{ovega@csufresno.edu}
\date{}


\begin{abstract}
We introduce and study embeddings of graphs in finite projective planes, and present related results for some families of graphs including complete graphs and complete bipartite graphs. We also make connections between embeddings of graphs and the existence of certain substructures in a plane, such as Baer subplanes and arcs.
\end{abstract}

\maketitle

\section{Introduction}

We are interested in studying how graphs may be embedded in finite planes. These embeddings are injective functions which preserve graph-incidence within the incidence relation of a finite projective plane and are analogous to the notion of embedding used in planar graphs and topological graph theory. Although planar graphs have been well-studied, embeddings of graphs into finite projective planes are relatively new. Thus, embeddings may give insight into the structure of finite projective planes.

We begin by introducing necessary terminology, and establish counting lemmas used in the later constructions and proofs. In Section \ref{sec:CombCBG} we prove our main results, determining which complete bipartite graphs can be embedded in a given plane, and by counting the number of embeddings of complete bipartite graphs of large order. Finally, in Section \ref{secsubplanes}, we establish bounds on the number of vertices and edges an embedded
complete bipartite graph can have in any given Baer subplane. We also establish conditions when the image of a complete bipartite graph must contain a non-Baer subplane.

\medskip

We now give some terminology and well-known results necessary for later sections. Most of the content of this section is `folklore'; any terminology not found here may be found in \cite{Bol98} (graph theory), or \cite{HP}  (geometry).

We will consider a graph $G=\left(V(G),E(G)\right)$ to be a finite set of vertices $V(G)$, paired with a set $E(G)$ of 2-subsets of $V(G)$ called edges.  We will exclude directed graphs, multigraphs, and graphs with loops. For simplicity, an edge $\{a,b\}\in E(G)$ will be denoted as simply $ab$, or equivalently $ba$. When context is clear, $V(G)$ and $E(G)$ will be referred to simply as $V$ or $E$. We will say that $a\in V$ is (graph) incident with $e = ab \in E$.  The vertices $a$ and $b$ are called the endpoints of edge $e = ab$. Graph-incidence is a symmetric relation on $V\times E$. A bijective function $\alpha:G\rightarrow G$ that preserves incidence will be said to be an automorphism of $G$. Homomorphisms/isomorphisms between graphs are defined accordingly. The group (under composition) of all automorphisms of $G$ is denoted by $Aut(G)$.

If $G=(V,E)$ is a graph such that there is a partition of $V$ into two classes $X$ and $Y$ such that for every $ab\in E$, $a\in X$ and $b\in Y$ or $b\in X$ and $a \in Y$, then $G$ is a bipartite graph. In this case $X$ and $Y$ are called the classes of $G$. If every point in $X$ is connected to every point in $Y$, then $G$ is a complete bipartite graph and is denoted $K_{|X|,|Y|}$. It is easy to see that $K_{m,n}$ is isomorphic to $K_{n,m}$, and so without loss of generality we assume $m\geq n$ unless otherwise specified.

\begin{definition}
A \textit{projective plane} $\pi=(\mathcal P,\mathcal L,\mathcal I)$ is a collection of points $\mathcal P$ and lines $\mathcal L$, paired with a symmetric incidence relation $\mathcal I\subseteq (\mathcal P\times \mathcal L) \cup (\mathcal L \times \mathcal P)$ such that: \\
\textbf{(a)} For every $P_1,P_2 \in \mathcal P$, there  exists a unique $\ell \in \mathcal L$ such that $P_1 \mathcal{I}\ell$ and $P_2 \mathcal{I}\ell$. \\
\textbf{(b)}  Every line is incident with at least three points. \\
\textbf{(c)}  For every  $\ell_1, \ell_2 \in \mathcal L$, there  exists a unique $P \in \mathcal P$ such that $\ell_1 \mathcal{I}P$ and $\ell_2 \mathcal{I}P$. \\
\textbf{(d)}  There exists four points in $\mathcal P$, no three of which are all incident with a common line in $\mathcal L$.
\end{definition}

It is known that every finite projective plane $\pi$ is associated to a natural number $q$, called the \textit{order} of $\pi$, such that $|\mathcal P|= |\mathcal L|= q^2+q+1$, and that the number of points on each line is exactly $q+1$, as well as the number of lines through any point is exactly $q+1$.  All known examples of projective planes have order equal to the power of some prime.  In fact, for every power of a prime there are examples of planes having that order, and if this number is at least $9$, and not a prime number, then non-isomorphic planes having that order exist. It is a long-standing open question as to whether there exists a finite projective plane with order that is not a prime power.

\begin{definition}
Let $\pi=(\mathcal P,\mathcal L,\mathcal I)$ be a projective plane. If there exists a projective plane $\pi_0=(\mathcal P_0,\mathcal L_0,\mathcal I_0)$ such that $\mathcal P_0\subseteq \mathcal P, \mathcal L_0 \subseteq \mathcal L,$ and $\mathcal I_0\subseteq \mathcal I$. Then $\pi_0$ is said to be a subplane of $\pi,$ denoted $\pi_0 \subseteq \pi$.
\end{definition}

\begin{theorem}[Bruck, see \cite{B55}]
Let $\pi$ be a projective plane of order $q$. If $\pi_0$ is a subplane of $\pi$ with order $n$, then either $n^2=q$ or $q\geq n^2+n$.
\end{theorem}

In the case when $n^2=q$ we say that $\pi_0$ is a Baer subplane of $\pi$. It is known that every line of $\pi$ must intersect $\pi_0$ in at least one point. Similarly, every point of $\pi$ is incident with at least one line of $\pi_0$.

\medskip

We are interested in studying how graphs will embed in a finite projective plane. First, we carefully define our notion of embedding.

\begin{definition}\label{defembedding}
An embedding $\phi$ of a graph $G=(V,E)$ into a projective plane $\pi=(\mathcal P, \mathcal L, \mathcal I)$ is an injective function
\[
\phi: V \rightarrow \mathcal P
\]
that induces, by preserving incidence, an injective function
\[
\overline{\phi}: E \rightarrow \mathcal L.
\]
If such an embedding exists we say that $G$ embeds in $\pi$,  and write $G\hookrightarrow \pi$.
\end{definition}

We remark that multigraphs are out of the range of the previous definition because we are identifying the \emph{unique} line through two points with a \emph{unique} edge through them. Also, since the points on any given line in a finite projective plane are not ordered, oriented graphs will not be considered. Finally, although Definition \ref{defembedding} is perhaps more precise and rigorous, when context is clear, we will denote an embedding $\phi$ of $G$ into $\pi$ as $\phi: G\hookrightarrow \pi$.

\medskip

The study of embeddings of graphs into projective planes has a short history. Although the original inspiration for the study of cycles in projective planes takes us back many decades (see \cite{LMV09} and \cite{LMV13} for a thorough historical narrative), considering the embedding of more complicated graphs has not been done in a serious and systematic way. Hence, this article intends to answer a few natural questions that quickly arise, and to present thought-provoking results that show unexpected connections between embeddings of complete bipartite graphs and well-known substructures of some projective planes.

Finally, we would like to remark that the work in \cite{LMV09} and  \cite{Vor12} deals with counting $k$-cycles in projective planes, for small values of $k$. Meanwhile, the pancyclicity (admitting embeddings of every possible cycle) of every every projective plane  is proved in \cite{LMV13}. Also, an algebraic approach to embeddings of cycles, and the study of wheels and gear graphs embedded into projective planes may be found in \cite{PVW13}.

\section{General results on embeddings of graphs}

Note that if $G$ is a graph, $\varphi \in Aut(G)$, and $\phi$ is an embedding of $G$ into $\pi$, then $\phi \circ \varphi$ is also an embedding of $G$ into $\pi$. However, these two embeddings `look the same' in $\pi$, as they use the same points and lines, and the incidence between the embedded vertices and edges is also the same. This inspires the following definitions.

\begin{definition}\label{defn&N}
Let $G$ be a graph, $\pi$ be a projective plane, and $\phi, \psi$ be embeddings of $G$ into $\pi$.\\
\textbf{(a)}  If $\psi=\phi \circ \varphi$, for some $\varphi \in Aut(G)$, we will say that $\phi$ and $\psi$ are equivalent embeddings of $G$ into $\pi$. \\
\textbf{(b)} We denote the number of embeddings of $G$ in $\pi$ by $N_{\pi}(G)$. The number of un-equivalent embeddings of $G$ into $\pi$ is denoted by $n_{\pi}(G)$.
\end{definition}

\begin{remark}\label{remN=nA}
Definition \ref{defn&N}\textbf{(a)} yields a free action of $Aut(G)$ on the set of embeddings of $G$ into $\pi$. Hence, $N_{\pi}(G)=n_{\pi}(G)|Aut(G)|$.
\end{remark}

\begin{definition}\label{defimage}
Given $\phi$, an embedding of a graph $G$ into a projective plane $\pi$, the incidence structure $(\phi(V(G)), \phi(E(G)), \I)$, where $\I$ is naturally induced by $\phi$ and the incidence in $G$, will be said to be an image of $G$ in $\pi$.
\end{definition}

\begin{remark}\label{remn=imag}
If $\phi$ and $\psi$ are two embeddings of a graph $G$ yielding the same image then $\phi^{-1} \psi$ is an automorphism of $G$. Hence, the number of images of $G$ in $\pi$ is equal to $n_{\pi}(G)$.
\end{remark}

We summarize Remarks \ref{remN=nA} and  \ref{remn=imag} in the following theorem.

\begin{theorem}\label{thmn=NAut}
Let $G$ be a graph embedded in a projective plane $\pi$. Then,
\[
N_{\pi}(G)=n_{\pi}(G)|Aut(G)|
\]
Moreover, $n_{\pi}(G)$ is equal to the number of images of $G$ in $\pi$.
\end{theorem}

\medskip

Theorem \ref{thmn=NAut} makes it fairly easy to count $n_{PG(2,2)}(G)$ for various graphs, such as small cycles and complete graphs. In all these cases, $n_{PG(2,2)}(G)$ is a multiple of $7 =2^2+2+1$. This observation yields the following easy generalization of Lemma 8 in \cite{LMV09}.

\begin{lemma}
Let $G$ be a graph embedded in a projective plane $\pi$ of order $q$ that admits a cyclic collineation group $H$ of order $m$ acting freely on the points of $\pi$. Then, $(m/d) \ |\  n_{\pi}(G)$, where $gcd(v(G), m)=d$.

The same result holds for a group acting freely on the lines of $\pi$, under the condition $gcd(e(G), m)=d$.
\end{lemma}

\begin{corollary}
If a projective plane $\pi$ of order $q$ admits a Singer cycle, and $G$ is a graph with $gcd( v(G), q^2+q+1)=1$ or  $gcd( e(G), q^2+q+1)=1$, then $(q^2+q+1) \ | \ n_{\pi}(G)$.
\end{corollary}

\begin{corollary}\label{corCandK}
If $gcd(s, q^2+q+1)=1$, then $q^2+q+1$ divides both  $n_{PG(2,q)}(C_s)$ and $n_{PG(2,q)}(K_s)$.
\end{corollary}

\medskip

\begin{remark}\label{remisomsamen}
If a graph $G$ is embedded in a projective plane $\pi$, then $G$ can also be embedded in any other plane that is isomorphic to $\pi$, or in one that has a subplane isomorphic to $\pi$.  It is also immediate that if a graph $G$ embeds in $\pi$ then so does every subgraph of $G$.
\end{remark}

One of the most fundamental questions about embeddings asks what graphs can be embedded in projective planes, and more specifically, what graphs can be embedded in a given plane. The first (trivial) bounds are given by the order $q$ of the plane $\pi$ that will `host' the graph $G$. For instance, $G$ cannot be embedded in $\pi$ if $v(G)$ or $e(G)$ are larger than $q^2+q+1$, as this is the number of lines, and points, in $\pi$. Similarly, a graph having a vertex of degree larger than $q+1$ cannot be embedded in $\pi$ because there are exactly $q+1$ lines through any given point of $\pi$. Thus we have shown.

\begin{lemma}
Let $G$ be a graph embedded in a projective plane $\pi$ of order $q$, then
\[
v(G) \leq q^2+q+1 \hspace{1in} e(G) \leq q^2+q+1
\]
and
\[
deg(v) \leq q+1
\]
for all $v \in V(G)$.
\end{lemma}

The bounds above are sharp. For the first two bounds we refer the reader to \cite{LMV13}, where given a projective plane $\pi$, a cycle using all lines and points of $\pi$ is constructed.  For an example reaching the latter bound, we observe that the set of lines through a point yields a trivial example.  For a less trivial example, we consider two distinct points $P$ and $Q$ in $\pi$, and $l$ a line that contains neither $P$ nor $Q$. Let $\{ l_1, l_2, \cdots , l_{n+1}  \}$ be the set of lines through $P$, $x_i = l_i \cap l$ for all $i =1,2,\cdots , q+1$, and let $m_i$ be the line through $Q$ and $x_i$, for all $i=1,2, \cdots , n+1$. Note that one of the lines through $P$ must go through $Q$.  This line will go through exactly one of the $x_i$'s.  Without loss of generality, assume it goes through $x_1$.  Then, the graph with vertices $\{P,Q\} \cup \{ x_i : \ i= 2,\ldots , q+1 \}$ and edges all the lines through $P$ or $Q$ has two vertices with degree exactly $q+1$.

\medskip

Embeddings of cycles have already been studied in \cite{LMV09} and \cite{LMV13}, thus the natural next family of graphs to consider would be that of complete graphs. Most precisely, we would like to know how large $n$ can be in order to embed $K_n$ into a plane of order $q$.  This problem boils down, quite naturally, to a well-studied object in finite geometry.

\begin{definition}
Let $\pi$ be a projective plane of order $q$. An $n$-arc in $\pi$ is a set of $n$ points of $\pi$ such that no three of them are collinear. If $q$ is odd, a $(q+1)$-arc in $\pi$ is said to be an oval. If $q$ is even, a $(q+2)$-arc in $\pi$ is called a hyperoval.
\end{definition}

\begin{remark}\label{removals}
No arc  can have cardinality larger than that of an oval. Not all projective planes contain ovals. For instance, there are  exactly four finite projective planes of order $16$ containing no hyperovals (see \cite{PRS96}).
\end{remark}

\begin{theorem}\label{thmcapcomplete}
Let $\pi$ be a projective plane of order $q$, and $n\in \mathbb{N}$. Then, $K_n \hookrightarrow \pi$ if and
only if there is an $n$-arc in $\pi$. Moreover, $n_{\pi}(K_n)$ is equal to the number of $n$-arcs in $\pi$.
\end{theorem}

\begin{proof}
Let $P,Q,R$ be any three points of $\pi$ corresponding to vertices of an embedded $K_n$.  Then, since the lines joining any $2$ of these points must corresponding to edges in $K_n$, it is necessary that $P,Q,R$ are not on a common line. \\
Conversely, if there is an $n$ arc in $\pi$, then we can connect all possible pairs of these points using distinct lines, yielding an embedding of $K_n$ into $\pi$.\\
Note that we have defined a bijection between images of $K_n$ in $\pi$ and $n$-gons. This finishes the proof.
\end{proof}

\begin{corollary}\label{corcomplete}
Let $\pi_q$ be a projective plane of order $q$, and $n\in \mathbb{N}$. If $K_n \hookrightarrow \pi_q$, then $q$ is even and $n \leq q+2$, or $q$ is odd and $n \leq q+1$.
\end{corollary}

The converse of Corollary \ref{corcomplete} is not true, as there are planes where ovals/hyperovals do not exist (see Remark \ref{removals} above). Note that this yields examples of planes that, even though they have the same order, do not admit embeddings of the same graphs.  Hence, the study of embeddings of graphs in projective planes could be used to classify planes. \\

We now look at conditions that would assure the embedding of a $K_n$ in a plane of order $q$. We want this condition to be independent from whether or not the plane contains ovals/hyperovals.

Let $\pi$ be a projective plane of order $q$ (think $q$ large). Since any projective plane contains a quadrangle, $K_4$ is embedded in $\pi$. The $6$ lines of $K_4$ cover at most $6(q+1)$ points. Thus if $q$ is large then there will be a point in $\pi$ that is not in any line of the embedding. It follows that we can add any of these points to $K_4$ to create an embedding of $K_5$ in $\pi$. Now we will have $10$ lines used in the embedding and thus at most $10(q+1)$ points covered by those lines. Then again, if $q$ is large,  there will be free points to create a $K_6$.

We can keep repeating the process above. The graph $K_n$ has $n(n-1)/2$ edges, and thus at most $n(n-1)(q+1)/2$ vertices are covered by it. Hence, if $q^2+q+1 \geq  n(n-1)(q+1)/2$ then $K_n$ can be embedded in $\pi$. Thus it is enough to ask $q \geq  n(n-1)/2$. We can phrase this result as follows.

\begin{theorem}
Let $G$ be the complete graph $K_{n}$, and choose $q \geq  n(n-1)/2$.  Then $G$ can be embedded in any projective plane $\pi$ of order $q$.
\end{theorem}

The following corollary is an immediate consequence of Remark \ref{removals}.

\begin{corollary}
Every graph $G$ can be embedded in any projective plane of order $q \geq  v(G)(v(G)-1)/2$.
\end{corollary}

It turns out that we can tighten these arguments up quite firmly.

\section{Complete bipartite graphs}\label{sec:CombCBG}

In the case of complete bipartite graphs, we can give a similar result.  We start with a natural construction that yields a complete bipartite graph and, therefore, any of its subgraphs.

\begin{theorem}\label{thmbipartembedd}
Let $G$ be a subgraph of the bipartite graph $K_{n,n}$. Then, $G$ can be embedded in any projective plane $\pi$ of order at least $n$.
\end{theorem}

\begin{proof}
Let $\pi$ be a projective plane of order at least $n$.\\
We choose a point $\mathcal{O}$ and any two distinct lines in $\pi$, $\ell$ and $m$, both through $\mathcal{O}$.  Let $\{ P_1, P_2, \cdots , P_n\}$ be a set of $n$ distinct points in $\ell \setminus \{\mathcal{O}\}$ and let $\{ Q_1, Q_2, \cdots , Q_n\}$ be a set of $n$ distinct points in $m \setminus \{\mathcal{O}\}$. \\
Let $l_{i,j}$ be the line of $\pi$ incident with $P_i$ and $Q_j$.  Since in a projective plane there is a unique line through any two points, the set $\{ l_{i,j} ; \ i,j = 1, 2, \cdots , n \}$ has exactly $n^2$ elements. It follows that $K_{n,n}$ is embedded in $\pi$, and thus so is $G$.
\end{proof}

Using ideas already discussed in Theorem \ref{thmbipartembedd} we can give a complete characterization for when complete bipartite graphs are embeddable in $\pi$.

\begin{theorem}\label{thmbipartPG2q}
The only complete bipartite graphs embedded in a projective plane $\pi$ of order $q$ are: \\
\textbf{(a)} $K_{n,m}$, where $n,m\leq q$, or \\
\textbf{(b)} $K_{1,q+1}$.
\end{theorem}

\begin{proof}
Consider $G=K_{n,q+1}$, for $n>1$, partitioned into independent sets $U$ and $V$ where $|U| =n$ and $|V|=q+1$. Let $\phi : G \hookrightarrow \pi$. \\
Every point  $P=\phi(a)$, where $a \in U$, is incident with exactly $q+1$ lines, each of which is used to form edges in $G$. This uses all the lines through $P$, and thus, there is no line through $P$ that can pass through a second point $Q=\phi(b)$, for $b \in U$. But $P$ and $Q$ must be collinear, yielding a contradiction. So, a bipartite graph can have at most $q$ elements in its independent sets. Theorem \ref{thmbipartembedd} assures that $K_{n,m}$, where $n,m\leq q$, can always be embedded in $\pi$. \\
For the other case, $K_{1,q+1}$ may be embedded in $\pi$ by considering a point and a line not through the
point. Also, for when $q$ is even, an oval and its nucleus yield a different embedding of the same graph.
\end{proof}

We now present results concerning the number of embeddings of complete bipartite graphs into planes. Most of these results will be in extremal cases, and will eventually lead us to consider subplanes in the following section.

Counting the embeddings of $G = K_{n,q}$, for $2\leq n < q$, is very complex, as the different embeddings can be `hosted' by many different configurations of points and lines; too many  to be counted effectively.  However, in specific cases one may get some interesting results.

\begin{lemma}\label{lemK1n}
Let $\pi$ be a projective plane of order $q$, and $1\leq n \leq q+1$. Then,
\[
n_{\pi}(K_{1,n}) =  \binom{q+1}{n}q^{n}(q^2+q+1)
\]
\end{lemma}

\begin{proof}
Let $G=K_{1,n}$, and $\phi : G \hookrightarrow \pi$. Let $v\in V(G)$ be the vertex in the singleton class. An embedding of $K_{1,n}$ is uniquely determined by the choice of $\phi(v)$, of which there are $q^2+q+1$ choices, and the choice of exactly one point on $n$ of the $q+1$ lines through $\phi(v)$. Once the $n$ lines to be used in the embedding are chosen, there are exactly $q^n$ ways to choose the $n$ additional points in $\phi(V(G))$.
\end{proof}

\begin{lemma}\label{lemqverticesbipart}
Let $\pi$ be a projective plane of order $q$. Consider $\phi: K_{n,q} \hookrightarrow \pi$, for $2\leq n \leq q$, and let $U$ and $V$, containing $n$ and $q$ vertices respectively, be the two classes of $K_{n,q}$. Then, all points in the embedding of $U$ are collinear.
\end{lemma}

\begin{proof}
Note that any given $P\in \phi(U)$ is adjacent to the elements in $\phi(V)$ via $q$ of the $q+1$ lines through $P$, forcing that none of the other $n-1$ elements in $\phi(U)$ is incident with any of these lines. Hence, $\phi(U)\subseteq m$, where $m$ is the only line through $P$ that does not intersect $\phi(V)$.
\end{proof}

\begin{theorem}\label{thmbigonecounting}
Let $\pi$ be a projective plane of order $q$. Consider $G = K_{n,q} \hookrightarrow \pi$, for $2\leq n \leq q$, where $V(G) =U\cup V$ with $|U| =n$ and $|V|=q$. Then, \\
\textbf{(a)} For $n=q$ we get:
\[
n_{\pi}(K_{q,q}) = \binom{q^2+q+1}{2}
\]
\textbf{(b)} If $n<q$, and every embedding of $V$ lies on a line, then
\[
n_{\pi}(K_{n,q}) = 2 \binom{q^2+q+1}{2} \binom{q}{n}
\]
\textbf{(c)} If $n=q-1$, then the points in the embedding of $V$ are collinear, and
\[
n_{\pi}(K_{q-1,q})  =  q^2(q+1)(q^2+q+1)
\]
\end{theorem}

\begin{proof}
Let $\phi : G \hookrightarrow \pi$, $2\leq n \leq q$, and $U, V$ as described above.\\
\textbf{(a)} Lemma \ref{lemqverticesbipart} implies that the $2q$ points in the embedding of $G$ must lie on two lines. It follows that the number of embeddings of $G$ into $\pi$ is given by the number of pairs of intersecting lines in $\pi$.\\
\textbf{(b)} We first choose two lines of $\pi$, and then choose one of them to be the one containing $\phi(U)$. The result follows from then choosing $n$ points, from the $q$ available points on this line, to create $\phi(U)$. \\
\textbf{(c)}  First notice that the result is valid for $q=2$ (and consistent with Lemma \ref{lemK1n}). For $q>2$, the value of $n_{\pi}(K_{q-1,q})$ will follow from \textbf{(b)}, as soon as we prove that the points of $\phi(G)$ lie on two lines.\\
Consider $P,Q\in \phi(V)$, and note that the remaining $q-2$ points of $\phi(V)$ cannot be on lines that connect $P$ or $Q$ with the $q-1$ points of $\phi(U)$. By inclusion/exclusion, the number of distinct points on these lines is exactly
\[
2(q(q-1)+1)-(q-1)^2 = q^2+1
\]
Also, the two unused points on the line containing $\phi(U)$ cannot be used, as this line passes through two points that are not neighbors in $G$. So, the remaining $q-2$ points of $\phi(V)$ have to be located among the
\[
(q^2+q+1) - (q^2+1)-2 = q- 2
\]
points remaining in $\pi$. These are exactly the points on $PQ$, different from $P$ and $Q$, not intersecting the line containing $\phi(U)$.
\end{proof}

In the proofs of the previous two results we have used lines connecting points in the same maximal independent set of the embedding of a complete bipartite graph. We want to generalize this idea, as it will help us to obtain several interesting results. Hence, we make the following definition:

\begin{definition}
Let $\pi$ be a projective plane and $\phi : G \hookrightarrow \pi$. If $P=\phi(v)$, $Q=\phi(w)$, where $v, w\in V(G)$ and $vw \notin E(G)$, then we will say that the line $PQ$ is a complement line of the embedding.
\end{definition}

We easily get the following two results. They will help us have a better understanding about the set of all complement lines of a complete bipartite graph.

\begin{lemma}\label{lemcomplementdisjoint}
Let $\phi : K_{m,n} \hookrightarrow \pi$ be an embedding of $K_{m,n}$ into plane $\pi$. Then the set of complement lines is disjoint from the set of embedded edges.
\end{lemma}
\begin{proof}
Let $U$ and $V$ be the maximal independent sets of $K_{m,n}$, and suppose $uv\in E(K_{m,n})$ is embedded to a complement line $\ell=\phi(uu_0)$, where $u_0\in U$. Then the edge $u_0v$ is embedded to $\ell$. Thus two edges are embedded to the same line which is a contradiction.
\end{proof}

\begin{lemma}\label{lemcomplementlinear}
Let $\phi : K_{m,n} \hookrightarrow \pi$ be an embedding of $K_{m,n}$ with $V(K_{m,n})=U\cup V$. Then $\phi(U)$ together with the set of complement lines of $\phi(U)$ form a linear space under the incidence relation of $\pi$.
\end{lemma}
\begin{proof}
By definition, every complement line must contain at least two points of $\phi(U)$. Since we embed into a projective plane, every pair of points may have at most one line through them.
\end{proof}

Theorem \ref{thmbigonecounting} \textbf{(b)} can be used to obtain $n_{\pi}(K_{q-n,q})$, as long as we know
that all embeddings of $K_{q-n,q}$ lay on two lines. Our next theorem provides conditions for this to happen.

\begin{theorem}\label{thmGlayontwolines}
Let $\pi$ be a projective plane of order $q$. Consider $G = K_{q-n,q} \hookrightarrow \pi$, for $2\leq q-n
\leq q$, where $V(G) =U\cup V$ with $|U| =q-n$ and $|V|=q$. If $q>n^2$, then the vertices of the embedding of
$G$ lay on two lines.
\end{theorem}

\begin{proof}
Lemma \ref{lemqverticesbipart} tells us that $|V|=q$ forces the existence of a line $\ell$ such that
$\phi(U)\subseteq \ell$. We let $A_1,A_2,\ldots , A_{n+1}$ be the $n+1$ points on $\ell \setminus \phi(U)$.

Let $P$ and $Q$ be two distinct points in $\phi(V)$. Since $PQ$ intersects $\ell$ we assume, without loss of
generality, that $A_{n+1} \in PQ$.  It follows that the elements of $\phi(V)$ must lay on $PQ$, or be one of
the points of intersection of the lines $PA_i$, and $QA_j$, for $1\leq i, j \leq n$. However, $A_i \notin
\phi(V)$, for all $1\leq i \leq n$, and thus at most $n(n-1)$ vertices of $\phi(V)$ are not on $PQ$. Note that
every complement line through $P$ must contain at most $n$ points of $\phi(V)$.

Now assume there is a point $R \in \phi(V)$ not collinear with $P$ and $Q$. The argument used above (now with $R$ instead of $Q$) forces $PQ$ to contain at most $n$ points of $\phi(V)$. Hence, $q=|\phi(V)|\leq n(n-1)+n=n^2$, which contradicts our hypothesis. So, the points of $\phi(V)$ must be collinear.
\end{proof}

\begin{remark}\label{remBaerembedding}
The bound $q>n^2$ in Theorem \ref{thmGlayontwolines} is sharp. For example, if $n=q-2$ and $q\geq 5$, then we get
\[
n_{\pi}(K_{q-2,q})  =  \frac{q^2(q^2-1)(q^2+q+1)}{2}
\]
But if $q=3$ we get  $n_{\pi}(K_{1,3}) = 13\cdot 4 \cdot 3^3=468$, and for $q=4$ we obtain $n_{\pi}(K_{2,4}) =
21 \cdot \binom{5}{2}\cdot 4!=5040$. These two values are not consistent with what Theorem
\ref{thmGlayontwolines} would have given us.

Moreover, if $q=n^2$ and $\pi$ has a Baer subplane $\mathcal{B}$ then we consider a line $\ell$ of $\mathcal{B}$ ($\ell$ is also seen as a `longer' line in $\pi$). We let $\phi(U)$ to be the points on $\ell \setminus \mathcal{B}$, and $\phi(V)$ to be the points on $\mathcal{B}\setminus \ell$. This selection of $\phi(U)$ and $\phi(V)$ yields an embedding of $K_{q-n,q}$ into $\pi$, with $q=n^2$, not having the points on $\phi(V)$ laying on a line.
\end{remark}

We will continue studying the interplay between subplanes and embeddings of complete bipartite graphs in the following section.

\section{Subplanes}\label{secsubplanes}

Note that the set of points in the embedding of a class of $K_{m,n}$ together with all its complement lines forms a partial plane. We are interested in studying whether this structure could ever be a subplane. Our hopes are, firstly, based on the following remark.

\begin{remark}\label{remsetupsubplane}
Consider $\pi$ a projective plane of order $q$, and assume $G=K_{m,n} \hookrightarrow \pi$, where $n=s^2+s+1$ and $m=q-s$, for some $s\in \mathbb{N}$. \\
There are at most $q+1-(q-s)=s+1$ complement lines per point in the maximal independent set of size $s^2+s+1$.
\end{remark}

\begin{theorem}[Erd\H os--De Brujin \cite{EdB48}]\label{Erdos}
Let $S=(P,L)$ be a finite linear space with $|P|=v$ and $|L|=b>1$. Then \\
\textbf{(a)}  $b\geq v$. \\
\textbf{(b)}  if $b=v$, any two lines of $S$ intersect at a point in $S$.\\
In case \textbf{(b)}, either one line has $v-1$ points and all others have two points, or $S$ is a finite projective plane.
\end{theorem}

The following lemma will help us in the study of these potential subplanes.

\begin{lemma}\label{BigLemma}
Let $\pi$ be a finite projective plane of order $q$, and $A=(\mathcal P,\mathcal L,\mathcal I)$ a set of points and lines of $\pi$ satisfying the axioms of a linear space. Assume that there is an integer $1<n\leq q$ such  that no point of $A$ is incident with more than $n+1$ lines in $A$. Then, \\
\textbf{(a)} If $|\mathcal P|= n^2+n+1$, then either $|\mathcal L| =1$, or  $A$ is a subplane of $\pi$. \\
\textbf{(b)}  If $|\mathcal P|>n^2+n+1$, then $|\mathcal L|=1$.
\end{lemma}

\begin{proof}
\textbf{(a)} If $|L|\neq 1$, by Theorem \ref{Erdos}, $A$ is either a finite projective plane or contains a line with $n^2+n$ points, with a single point not on the line. Suppose the latter, then there would be $n^2+n$ lines through the point not on the line, contradicting the assumption that each point may have at most $n+1$ lines of $A$ incident with it. Hence, $A$ is a finite projective plane if $|L|\neq 1$. If $|L|=1$, all points must be collinear, since every two points are collinear in a linear space.\\
\textbf{(b)} Follows from Theorem \ref{Erdos}\textbf{\textit{(a)}}.
\end{proof}

\begin{theorem}\label{SubplaneBound}
Let $\pi$ be a projective plane of order $q$ and let $n>1$ be a natural number. Then, any embedding of $K_{q-n, \ n^2+n+1}$ maps the maximal independent set of size $n^2+n+1$ to either a subplane of order $n$, or to points on a line.
\end{theorem}

\begin{proof}
Note that, implicitly, we are assuming $q>n^2+n$, as this is needed for $K_{q-n, \ n^2+n+1}$ to embed into $\pi$. \\
Let $P$ be an embedded vertex of the maximal independent set of size $n^2+n+1$. Because of Remark \ref{remsetupsubplane}, $P$ is incident with at most $n+1$ complement lines. It follows that the set of embedded vertices of the maximal independent set of size $n^2+n+1$, together with their complement lines and the incidence relation inherited from $\pi$ satisfy the hypothesis of Lemma \ref{BigLemma}. The result follows.
\end{proof}

\begin{corollary}
Let $\pi$ be a projective plane of order $q$ and let $n>1$ be a natural number. Then any embedding of $K_{s,t}$, where $s\geq q-n$, and $t> n^2+n+1$ maps the maximal independent set of size $t$ to points on a line.
\end{corollary}

\begin{proof}
Since $t> n^2+n+1$, there are at least two embeddings of $K_{q-n, \ n^2+n+1}$ into $\pi$, sharing a  $K_{q-n, \ n^2+n}$. We denote these embeddings $G_1$ and $G_2$. We now apply Theorem \ref{SubplaneBound} for each of these graphs. If the two classes of size $n^2+n+1$ are subplanes we get a contradiction, as the partial plane on $n^2+n$ points (the intersection of the two subplanes) can be extended in a unique way to a projective plane, as the missing point must be the intersection of two lines that are parallel in the subplane. It follows that, without loss of generality, $G_1$ has its independent set of cardinality $n^2+n+1$ contained in a line. Hence, $G_2$ must have the same property, as $G_1$ and $G_2$ share a $K_{q-n, \ n^2+n}$. Moreover, since these lines will share at least two points, they are the same line.
\end{proof}

\begin{corollary}\label{corspecialcase}
Let $\pi$ be a projective plane of order $q$, and let $n, s, t \in \mathbb{N}$ be such that $q=(n+1)^2$, and $q \geq s, t > q-n$. Then,
\[
n_{\pi}(K_{s,t}) = 2 \binom{q^2+q+1}{2} \binom{q}{s} \binom{q}{t}
\]
for $s\neq t$, and
\[
n_{\pi}(K_{s,s}) = \binom{q^2+q+1}{2} \binom{q}{s}^2
\]
\end{corollary}

\begin{proof}
If $q=(n+1)^2$, then $q-n =n^2+n+1$. It follows that $q\geq s,t > n^2+n+1$, and thus each maximal independent set of $K_{s,t}$ must embedded into a line. Hence, the number of images of $K_{s,t}$ in $\pi$ is thus obtained by counting how to choose two distinct lines in $\pi$ and then how to choose points on them.
\end{proof}

\begin{remark}
If $t$ is a Mersenne prime, or $t+1$ is a Fermat prime, then the pair $(t, (t+1)^2)$ could serve as the $(n,q)$ pair required in the hypothesis of Corollary \ref{corspecialcase}. An example of a pair not covered already would be $(n,q)= (8, 3^4)$.
\end{remark}

\begin{definition}
Let $G=(V,E)$ be a graph and $G_0=(V_0,E_0)$ be a subgraph of $G$. If, given $v_1,v_2\in V_0$,  $v_1v_2 \in E$ implies $v_1v_2\in E_0$, then $G_0$ is said to be a full subgraph of $G$.
\end{definition}

\begin{lemma}\label{lematmostq^2+1inBaer}
Let $\Pi$ be a projective plane of order $q^2$, and $\pi$ be a Baer subplane of $\Pi$.  Then, an embedding of a maximal complete graph $K_{q^2,q^2}$ into $\Pi$ can contain at most either: \\
\textbf{(a)}  $q^2$ points of a Baer subplane, for $q>2$, or \\
\textbf{(b)} $q^2+1$ points of a Baer subplane, for $q=2$.
\end{lemma}

\begin{proof}
Let $q>2$ and $\phi : K_{q^2,q^2} \hookrightarrow \Pi$. Let $K$ be the full subgraph of $K_{q^2,q^2}$ on the set of vertices which are embedded to points of $\pi$.\\
Suppose $\phi(K_{q^2,q^2})$ contains $q^2+1$ points of a Baer subplane. It follows that $K$ must be a complete bipartite graph, and thus Theorem \ref{thmbipartPG2q} can be applied to $K$ and $\pi$. Hence, since  $q^2=q+1$ is impossible, $K$ must be of the form $K_{s,t}$ with $2\leq s,t \leq q$. Then, there is an integer $n$ such that $2\leq n \leq q^2-1$ such that $K$ is isomorphic to $K_{n,q^2+1-n}$. It follows that this graph contains exactly $n(q^2+1-n)$ edges. This number must minimize at the endpoints of the interval $[2,q^2-1]$, as $n(q^2+1-n)$ is quadratic in $n$ with negative leading coefficient. Hence, the minimum number of edges is $2(q^2-1) =  (q^2+q+1)+ (q^2-q-3)$, but this number is larger than $q^2+q+1$ because $q>2$, a contradiction. If $q=2$  the result follows similarly.
\end{proof}

Lemma \ref{lematmostq^2+1inBaer} yields a string of corollaries, the first one of them being the most immediate.

\begin{corollary}
Any complete bipartite graph $K_{m,n}$ embedded in a plane of order $q^2$ can contain at most $q^2+1$ points of a Baer subplane.
\end{corollary}

\begin{corollary}
Let $\Pi$ be a plane of order $q$. If a subplane $\pi$ of order $n$ contains a cycle of an embedded $K_{q,q}$, then it contains an embedded $K_{n,n}$.
\end{corollary}

\begin{proof}
If such a cycle exists, at least two points from either class of $K_{q,q}$ must be in $\pi$. Since $K_{q,q}$ is maximal, we know that points in each partition are on the same line. Hence, the cycle implies that $\pi$ contains both such lines and thus contains the full sub-embedding of $K_{q,q}$
\end{proof}

\begin{corollary}
Let $\pi$ be a projective plane of order $q$ and $\phi: K_{q,q}\hookrightarrow \pi$. Then, $\phi(K_{q,q})$ can contain at most $2n$ points of a subplane of $\pi$ of order $n$.
\end{corollary}

\begin{proof}
Since the image of $K_{q,q}$ under an embedding is on two lines (Theorem \ref{thmbigonecounting}), any intersection of the subplane and two lines of the parent plane contains at most $2n+1$ points, but in the maximal case, the point of intersection of the two lines is not an embedded vertex, and thus there may at most be $2n$ vertices embedded into the subplane.
\end{proof}

\begin{corollary}
An embedding of $K_{q^2,q^2}$ into a projective plane of order $q^2$, $\Pi$, must contain at least $q$ lines of every Baer subplane of $\Pi$.
\end{corollary}

\begin{proof}
Let $\pi$ be an arbitrary Baer subplane of $\Pi$. Since there are $q^4+q^2+1$ lines in $\Pi$, and $q^2+q+1$ lines in $\pi$, there are $q^4-q$ lines of $\Pi$ not in $\pi$. Since $K_{q^2,q^2}$ contains $q^4$ edges, there must be $q$ lines which are in $\pi$.
\end{proof}

\bigskip

The ideas in this article have the potential to be very useful in the study, and classification, of finite projective planes. This is based on Remarks \ref{remisomsamen}, \ref{removals} and \ref{remBaerembedding}, and on the fact that most of the results in Section \ref{secsubplanes} depend on the existence of Baer subplanes, which are known to be present in certain (but not all) planes, and occur in different numbers depending on the plane.

It would be interesting to continue the work on embeddings by looking for connections between the existence of certain configurations (e.g. Pappus, Desargues, etc) in a given plane and the types of graphs that can be embedded in such plane. Also, it has become evident that it is necessary to understand how the collineation group of a plane connects, if it does at all, with the automorphism group of a graph that embeds in it.


\end{document}